\documentclass[authoryear,a4paper,12pt]{elsarticle}
\usepackage{graphicx}
\usepackage{amsmath,amsthm,amssymb,mathrsfs,mathtools}
\usepackage[OT1]{fontenc} 
\usepackage{etoc}

\usepackage{fancyhdr}

\def\subsectiontitle{}
\def\subsubsectiontitle{}
\usepackage{mathpazo,mathabx}
\usepackage{enumitem}
\usepackage[margin=1in]{geometry}
\usepackage[colorlinks,breaklinks=true,pagebackref]{hyperref}
\usepackage{breakcites}
\newcommand{\abs}[1]{\left| #1 \right|} 

\usepackage{xcolor}
\AtBeginDocument{%
	\hypersetup{%
    linkcolor={red!50!black},%
    citecolor={blue!50!black},%
    urlcolor={blue!80!black}%
}%
}%

\usepackage[nameinlink,noabbrev,sort,capitalise]{cleveref}
\makeatletter
\def\ps@pprintTitle{%
 \let\@oddhead\@empty
 \let\@evenhead\@empty
 \def\@oddfoot{\emph{First draft: October 31, 2017\hfill{This draft: \today}}}%
 \let\@evenfoot\@oddfoot}
\makeatother
\usepackage{etoolbox}
\patchcmd{\pprintMaketitle}
  {\fi\hrule}
  {\fi\ifvoid\extrainfobox\else\unvbox\extrainfobox\par\vskip10pt\fi\hrule}
  {}{}

\newsavebox\extrainfobox

\newtheorem{prop}{Proposition}
\crefname{prop}{Proposition}{Propositions}

\newtheorem{thm}{Theorem}
\crefname{thm}{Theorem}{Theorems}

\newtheorem{cor}{Corollary}[thm]
\crefname{cor}{Corollary}{Corollaries}

\newtheorem{lem}{Lemma}
\crefname{lem}{Lemma}{Lemmas}

\newtheorem{ass}{Assumption}
\crefname{ass}{Assumption}{Assumptions}

\newtheorem{defi}{Definition}
\crefname{defi}{Definition}{Definitions}

\theoremstyle{remark}

\theoremstyle{definition}

\crefname{eg}{Example}{Examples}

\crefname{problem}{Problem}{Problems}

\newcommand{\overbar}[1]{\mkern 1.5mu\overline{\mkern-1.5mu#1\mkern-1.5mu}\mkern 1.5mu}

\usepackage{indentfirst}

\allowdisplaybreaks

\let\oldfootnote\footnote
\renewcommand\footnote[1]{\oldfootnote{\hspace{.4mm}#1}}

\makeatletter
\renewenvironment{proof}[1][\proofname] {\par\pushQED{\qed}\normalfont\topsep6\p@\@plus6\p@\relax\trivlist\item[\hskip\labelsep\bfseries#1\@addpunct{.}]\ignorespaces}{\popQED\endtrivlist\@endpefalse}
\makeatother

\let\oldFootnote\footnote
\newcommand\nextToken\relax

\renewcommand\footnote[1]{%
    \oldFootnote{#1}\futurelet\nextToken\isFootnote}

\newcommand\isFootnote{%
    \ifx\footnote\nextToken\textsuperscript{,}\fi}

\begin{document}

\begin{abstract}
	Let $\mathcal{V}$ be the set of all combinations of expected value of finite objective functions from designing information. I showed that $\mathcal{V}$ is a compact and convex set implementable by signal structures with finite support when unknown states are finite. Moreover, $\mathcal{V}(\mu)$ as a correspondence of prior is continuous. This result can be applied to develop a concavification method of Lagrange multipliers that works with general constrained optimization. It also provides tractability to a wide range of information design problems.
\end{abstract}

\begin{frontmatter}
	\title{Information design possibility set}
	\author[wz]{Weijie Zhong}
	\address[wz]{Columbia University}
	\ead{wz2269@columiba.edu}
\end{frontmatter}

\section{Introduction}
Let $X$ be a non-empty finite set (state space). $\Delta(X)\in \mathbb{R}^{\abs{X}}$ is the set of all probability measure on $X$. Let $\mu$ denote elements in $\Delta(X)$. $\Delta^2(X)$ is the set of all probability measures (standard Borel measurability) on $\Delta(X)$. Let $P$ denote elements in $\Delta^2$(X). Let $\left\{ V^i \right\}_{i=1}^{n}$ be a finite set of continuous function on $\Delta(X)$. Let $f:\mathbb{R}^{n}\to \mathbb{R}$ be a continuous function. Let $D(\mu):\Delta(X)\to\mathbb{R}^{n}$ be closed valued.\par
My objective is to solve the following constrained maximization problem:
\begin{align}
	\sup_{P\in \Delta^2(X)}&f\left( E_P[V^1],\dots,E_P[V^n] \right)\label{eqn:P}\\
	\mathrm{s.t.}\ &
	\begin{dcases}
		\left( E_{P}[V^1],\dots,E_{P}[V^n] \right)\in D(\mu)\\
		E_P[\nu]=\mu
	\end{dcases}\notag
\end{align}
Suppose $n=1$ and $D\equiv \mathbb{R}$, then \cref{eqn:P} can be solved by concavifying $V^1(\mu)$ (\cite{kamenica2011bayesian}, \cite{aumann1995repeated}). And \cref{thm:cara} implies that it is without loss to consider optimal information structure involving signal number no more than $\abs{X}$. This gives tractability both analytically and computationally. However, even when $n=2$, with a general $f$ or a nontrivial constraint $D$, concavification no longer works and we might need to search over an infinite dimensional space to solve \cref{eqn:P}.\par
To solve \cref{eqn:P}, I studied the set of all possible combinations of expected valuation that can be implemented by designing information $P$. In \cref{sec:thm}, I proved a two-step concavification method: First, the information design possibility set itself can be implemented by combining finite number of information structures that implement its extreme points. Second, each extreme point can be implemented by concavifying a linear combination of $V^i$'s, thus involves only finite number of signals.\par

The general concavification method developed in this paper can be applied to a wide range of information design problems. In \cref{sec:app}, I first provide two applications in static information acquisition and dynamic information acquisition to show that the optimal solutions have a nice Lagrange multiplier characterization. Then I provide an application of persuading receivers with outside options to illustrate how the Lagrange characterization can simplify the optimization problem. Finally I provide an application of \cref{lem:1} in a setup of screening using information structures, to illustrate how the dimensionality of the problem can be reduced to make the problem tractable.

\section{Information possibility set}
Notations used: given a convex set $C$, $\mathrm{ext}(C)$ is set of extreme points of $C$, $\mathrm{ext}_k(C)$ is set of $k$-extreme points of $C$ \footnote{$\mathrm{ext}(C)=\bigcup_{k< n}\mathrm{ext}_k(C)$ and $C=\bigcup_{k\le n}\mathrm{ext}_k(C)$.}. $\mathrm{exp}(C)$ is set of exposed points of $C$. $F(C)$ is set of faces of $C$.
\begin{defi}
	Information possibility set $\mathcal{V}(\mu)\in\mathbb{R}^n$ is defined as:
	\begin{align*}
		\mathcal{V}(\mu)=\left\{ \left( E_P[V^1],\dots,E_P[V^{n}] \right)\bigg| P\in \Delta^2(X),E_P[\nu]=\mu \right\}
	\end{align*}
\end{defi}

\begin{lem}
	$\forall \mu$, $\mathcal{V}(\mu)$ is a compact and convex set. $\forall v \in \mathrm{ext}_k(\mathcal{V}(\mu))$, there exists $P\in \Delta^2(X)$ such that:
	\begin{align*}
		\begin{cases}
			v=\left( E_P[V^1],\dots, E_P[V^n] \right)\\
			\abs{\mathrm{supp}(P)}\le(k+1)\abs{X}
		\end{cases}
	\end{align*}
	\label{lem:1}
\end{lem}

\begin{proof}
	First of all, we prove that $\mathcal{V}(\mu)$ is compact and convex.
	\begin{itemize}
		\item Boundedness: $\forall P\in \Delta^2(X)$, $\min_{\mu\in \Delta(X)}V^i(\mu)\le E_P[V^i]\le\max_{\mu\in \Delta(X)}V^i(\mu)$. Therefore, $\forall v\in \mathcal{V}(\mu)$, it is bounded from $0$ by $\max_{\mu\in \Delta(X),i}\abs{V^i(\mu)}$ by sup norm. So $\mathcal{V}(\mu)$ is a bounded set.
		\item Convexity: $\forall v_1,v_2\in \mathcal{V}(\mu)$, there exists $P_1,P_2\in \mathcal{D}^2(X)$ s.t. $v_i=\left( E_{P_i}[V^1],\dots,E_{P_i}[V^{n}] \right)$. Since $\Delta^2(X)$ is a linear space and expectation operator is linear functional, $\forall \beta\in[0,1]$, $P_{\beta}=\beta P_1+(1-\beta) P_2\in \Delta^2(X)$ and:
			\begin{align*}
				v_{\beta}=& \left( E_{P_{\beta}}[V^1],\dots,E_{P_{\beta}}[V^{n}] \right)\\
				=&\beta\left( E_{P_1}[V^1],\dots,E_{P_1}[V^{n}] \right)+(1-\beta)\left( E_{P_2}[V^1],\dots,E_{P_2}[V^{n}] \right)\\
				=&\beta v_1+(1-\beta) v_2
			\end{align*}
			Therefore, $\beta v_1+(1-\beta)v_2\in \mathcal{V}(\mu)$ so $\mathcal{V}(\mu)$ is a convex set.
		\item Closeness: $\Delta(X)$ is a finite dimensional simplex. If we consider Prokhorov metric on $\Delta^2(X)$, then $\Delta^2(X)$ is a complete and separable space (Theorem 6.8 of \cite{billingsley2013convergence}). Now since $\Delta(X)$ is compact, by \cref{thm:prok}, $\Delta^2(X)$ is a compact, complete and separable space with Prokhorov metric. Prokhorov metric induces a topology equivalent to weak$^*$ topology(by Theorem 6.8 of \cite{billingsley2013convergence}). So $\forall v_k\in \mathcal{V}(\mu)$, if $v_k\to v$, then consider the sequence $P_k$ such that $v_k=E_{P_k}\left[ (V^i) \right]$. By compactness of $\Delta^2(X)$, pick a subsequence $P_k\xrightarrow{w-^*} P$. Then $\forall V^i$, since $V^i$ is continuous, $E_{P_k}[V^i]\to E_{P}[V^i]$. So $v\in \mathcal{V}(\mu)$ and $\mathcal{V}(\mu)$ is a closed set.
		\item Compactness: $\mathcal{V}(\mu)$ is a finite dimensional bounded and closed set, so it is compact.
	\end{itemize}

	$\forall v\in \mathrm{ext}_k(\mathcal{V}(\mu))$, $v$ is an interior point of a $k$-dimensional face $F$ of $\mathcal{V}(\mu)$. Then by \cref{thm:krein}, $v\in \mathrm{conv}\left( \mathrm{ext}(F) \right)$. By \cref{thm:cara}, there exists $\left\{ v_j \right\}_{j=1}^{k+1}\subset \mathrm{ext}(F)$ and $\sum \pi_j=1$ s.t. $\sum \pi_j v_j=v$. By \cref{lem:face}, $\left\{ v_j \right\}\subset \mathrm{ext}(\mathcal{V}(\mu))$. The next step is to prove that $\forall j$, there exists $P_j\in \Delta^2(X)$ s.t. $v_j=\left( E_P[V^1],\dots,E_P[V^n] \right)$ and $\abs{\mathrm{supp}(P_j)}\le\abs{X}$. 

\begin{lem}
	$\forall \mu$, $\forall v\in \mathrm{exp}(\mathcal{V}(\mu))$, $\exists P\in \Delta^2(X)$ and $\abs{\mathrm{supp}(P)}\le \abs{X}$ s.t. $v=E_P\left[ (V^i) \right]$.
	\label{lem:exp}
\end{lem}
\begin{proof}
	By definition of exposed points, there exists a linear function $l\in L(\mathbb{R}^{n})$ s.t. 
	\begin{align*}
		l(v)>l(v')\ \forall v'\in \mathcal{V}(x),v'\neq v
	\end{align*}
	In finite dimensional space, a linear function $l(v)$ can be equivalently written as $\sum \lambda_i v_i+c$. Consider the following maximization problem:
	\begin{align}
		\sup_{P\in \Delta^2(X)}&E_P\left[ \sum\lambda_i V^i+c \right]\label{eqn:bp}\\
		\mathrm{s.t.}\ &E_P[\nu]=\mu\notag
	\end{align}
	By \cref{thm:cara}, \cref{eqn:bp} can be solved by convexifying the graph of $\sum \lambda_i V^i(\mu)+c$. The maximum is achieved by a $P$ s.t. $\abs{\mathrm{supp}(P)}\le \abs{X}$. Of course $E_P[(V^i)]\in \mathcal{V}(\mu)$. Then by definition of $l$, $l(v)\ge E_P\left[ \sum \lambda_i V^i+c \right]$. On the other hand, there exists $P'\in \Delta^2(X)$ s.t. $v=E_{P'}\left[ (V^i) \right]$, by optimality of $P$, $l(v)\le E_P\left[ \sum \lambda_iV^i+c \right]$. Therefore, since $v$ is the unique element in $\mathcal{V}(\mu)$ achieving $l(v)$, we have $E_P[(V^i)]=v$ and $\abs{\mathrm{supp}(P)}\le \abs{X}$. 
\end{proof}

$\forall v^j\in \mathrm{ext}(\mathcal{V}(\mu))$, by \cref{thm:stra}, there exists $\left\{ v^{jl} \right\}_{l=1}^{\infty}\subset \mathrm{exp}(\mathcal{V}(\mu))$ and $\lim_{l\to\infty}v^{jl}=v^j$. By \cref{lem:exp}, there exists $P^{jl}\in \Delta^2(X)$ s.t. $\abs{\mathrm{supp}(P^{jl})}\le\abs{X}$ and $v^{jl}=E_{P^{jl}}\left[ (V^i) \right]$. Now each $P^{jl}$ can be represented as $\left( p^{jl}_t,\mu^{jl}_t \right)_{t=1}^{\abs{X}}\in\mathbb{R}^{2\abs{X}}$, where:
\begin{align*}
	\begin{cases}
		\sum_t p^{jl}_t=1\\
		\sum_t p^{jl}_t\mu^{jl}_t=\mu\\
		\sum_t p^{jl}_t V^i(\mu^{jl}_t)=v^{jl}_i\ \forall i
	\end{cases}
\end{align*}
Since $\left( p^{jl}_t,\mu^{jl}_t \right)$ is in finite dimensional vector space, there exists a subsequence converging to $\left( p^{j}_t,\mu^j_t \right)$ when $l\to \infty$. Therefore, since $V^i$ is each continuous, it is easy to verify that:
\begin{align*}
	\begin{cases}
		\sum_tp^{j}_t=1\\
		\sum_tp^j_t\mu^j_t=\mu\\
		\sum_tp^j_tV^i(\mu^j_t)=v^j_i\ \forall i
	\end{cases}
\end{align*}
Therefore, $v^j$ is implemented by $P^j\in \Delta^2(X)$ and $\abs{\mathrm{supp}(P^j)}\le\abs{X}$. So $P=\sum \pi_j P^j\in \Delta^2(X)$ and $\abs{\mathrm{supp}(P)}\le (k+1)\cdot\abs{X}$. By linearity of expectation operator, $E_{P}\left[ (V^i) \right]=\sum \pi_j E_{P^j}\left[ (V^i) \right]=\sum \pi_j v^j=v$.
\end{proof}

\begin{lem}
	Correspondence $\mathcal{V}:\Delta(X)\to \mathbb{R}^n$ is continuous. $\mathrm{Gr}(\mathcal{V})$ is convex and compact.
	\label{lem:closegraph}
\end{lem}
\begin{proof}${}$
	\begin{itemize}
		\item Boundedness: $\Delta(x)$ is a bounded set. $\forall \mu\in \Delta(X)$, $\mathcal{V}$ is uniformly bounded by radious $\max_{\mu\in \Delta(X),i}\abs{V^i(\mu)}$ by sup norm. So $\mathrm{Gr}(\mathcal{V})$ is bounded.
		\item Convexity: $\forall (\mu_1,v_1),(\mu_2,v_2)\in \mathrm{Gr}(\mathcal{V})$. $\forall \alpha\in[0,1]$. Since $\Delta(X)$ is convex, $\mu_{\alpha}=\alpha \mu_1+(1-\alpha)\mu_2\in \Delta(X)$. Now we prove that $v_{\alpha}=\alpha v_1+(1-\alpha)v_2\in \mathcal{V}(\mu_{\alpha})$. By definition, there exists $P_1,P_2\in \Delta^2(X)$ s.t. $E_{P_1}[\left( V^i \right)]=v_1$, $E_{P_1}[\nu]=\mu_1$ and $E_{P_2}\left[ (v^i) \right]=v_2$, $E_{P_2}[\nu]=\mu_2$. Define $P_{\alpha}=\alpha P_1+(1-\alpha)P_2$, then by linearity of expectation operator, $E_{P_{\alpha}}[\nu]=\alpha E_{P_1}[\nu]+(1-\alpha)E_{P_2}[\nu]=\mu_{\alpha}$. $E_{P_{\alpha}}\left[ (V^i) \right]=\alpha E_{P_1}[(V^i)]+(1-\alpha)E_{P_2}[(V^i)]=v_{\alpha}$. Therefore, $v_{\alpha}\in \mathcal{V}(\mu_{\alpha})$. So $(\mu_{\alpha},v_{\alpha})\in \mathrm{Gr}(\mathcal{V})$.
		\item Closedness: $\forall \left\{ (\mu_j,v_j) \right\}\subset \mathrm{Gr}(\mathcal{V})$, suppose $\mu_j\to \mu$, $v_j\to v$. Want to show that $\mu\in \Delta(X)$ and $v\in \mathcal{V}(\mu)$. First of all, since $\Delta(X)$ is complete, $\mu\in \Delta(X)$. Now by \cref{lem:1}, there exists $\left( p_j,\nu_j \right)$ such that:
			\begin{align*}
				\begin{cases}
					\sum_{k=1}^{(n+1)\abs{X}} p_j^k=1\\
					\sum_{k=1}^{(n+1)\abs{X}} p_j^k \nu_j^k=\mu_j\\
					\sum_{k=1}^{(n+1)\abs{X}} p_j^k V^i(\nu_j^k)=v_j^i
				\end{cases}
			\end{align*}
			Now since $p_j\in \Delta\left( (n+1)\abs{X} \right)$ and $\nu_j\in \Delta(X)$ are both compact spaces. Consider stadard Euclidean metric on product space $\Delta\left( (n+1)\abs{X} \right)\times \Delta(X)^{(n+1)\abs{X}}$, it is also compact. Therefore there exists convergincing subsequence $p_j\to p$ and $\nu_j^k\to \nu^k$. Then
			\begin{align*}
				\begin{cases}
					\sum_{k=1}^{(n+1)\abs{X}} p^k=\lim_{j\to\infty}\sum_{k=1}^{(n+1)\abs{X}} p_j^k=1\\
					\sum_{k=1}^{(n+1)\abs{X}} p^k \nu^k=\lim_{j\to\infty}\sum_{k=1}^{(n+1)\abs{X}} p_j^k \nu_j^k=\lim_{j\to\infty}\mu_j=\mu\\
					\sum_{k=1}^{(n+1)\abs{X}} p^k V^i(\nu^k)=\lim_{j\to\infty}\sum_{k=1}^{(n+1)\abs{X}} p_j^k V^i(\nu_j^k)=\lim_{j\to\infty}v^i_j=v^i
				\end{cases}
			\end{align*}
			Therefore, $(p,\nu)$ implements $v$ at $\mu$. So $v\in \mathcal{V}(\mu)$.
		\item Compactness: Since $\mathrm{Gr}(\mathcal{V})$ is closed and bounded, it is compact.
		\item Continuity: Since $\mathrm{Gr}(\mathcal{V})$ is compact, $\mathcal{V}(\mu)$ is upper hemicontinuous. Now we only need to show lower hemicontinuity. $\forall \left( \mu_m \right)\subset \Delta(X)$, $\mu_m\to \mu\in \Delta(X)$. $\forall v\in \mathcal{V}(\mu)$. By \cref{lem:1}, $v$ is impelemnted by $(p,\nu)$ with support size $(n+1)\abs{X}$. There exists a stochastic matrix $q_{jk}$ such that:
			\begin{align*}
				&\begin{cases}
					\nu_j=\frac{1}{\sum_k \mu_k q_{jk}}(\mu_1q_{j1},\dots,\mu_{-1}q_{j,-1})\\
					p_j=\sum_k \mu_k q_{jk}
				\end{cases}\\
				\implies&
				\begin{dcases}
					\frac{\partial p_j}{\mu_k}=q_{jk}\\
					\frac{\partial \nu_{jl}}{\partial \mu_k}=
					\begin{dcases}
						\frac{p_jq_{jl}-\mu_l q_{jl}^2}{p_j^2}&\text{when $k=l$}\\
						\frac{-\mu_lq_{jl}q_{jk}}{p_j^2}&\text{when $k\neq l$}
					\end{dcases}
				\end{dcases}
			\end{align*}
			Therefore, since each $p_j>0$, when $\mu_m$ is sufficiently close to $\mu$, corresponding $(p_m,\nu_m)$ will be bounded from $(p,\nu)$ by $\abs{\mu-\mu_m}$. By continuity of $V^i$, $v_m=\left(\sum p_m V^i(\nu_m)\right)\to \left(\sum p V^i(\nu)\right)=v$. Therefore, $\mathcal{V}(\mu)$ is both upper hemicontinuous and lower hemicontinuous.
	\end{itemize}
\end{proof}

\section{Main theorem}
\label{sec:thm}
\subsection{Existence and finite support}
\begin{thm}
	Let $X$ be a non-empty state space, $\left\{ V^i \right\}_{i=1}^{n}\subset C\Delta(X)$, $f\in C\mathbb{R}^n$. $\forall \mu\in \Delta(X)$, suppose $\mathcal{V}(\mu)\bigcap D(\mu)\neq \emptyset$, then there exists $P^*\in \Delta^2(X)$ solving \cref{eqn:P} and $\abs{\mathrm{supp}(P^*)}\le (n+1)\cdot\abs{X}$.
	\label{thm:1}
\end{thm}
\begin{proof}
	By definition of $\mathcal{V}(\mu)$, \cref{eqn:P} is equivalent to the following problem:
	\begin{align}
		\sup_{v\in D\bigcap \mathcal{V}(\mu)} f(v)\label{eqn:reduced}
	\end{align}
	By \cref{lem:1}, $\mathcal{V}(\mu)$ is a compact set. Then $\mathcal{V}(\mu)\bigcap D(\mu)$ is compact and non-empty. By Wierestrass's theorem, there exists  $v^*\in \mathcal{V}(\mu)\bigcap D(\mu)$ solving \cref{eqn:reduced}. Then by \cref{lem:1}, there exists $P^*\in \Delta^2(X)$ s.t. $v^*=\left( E_{P^*}[V^1],\dots, E_{P^*}[V^n] \right)$ and $\abs{\mathrm{supp}(P^*)}\le (n+1)\cdot \abs{X}$. Therefore, $P^*$ solves \cref{eqn:P}.
\end{proof}

\subsection{Necessary condition of optimizer}
\begin{thm}
	Let $X$ be a non-empty state space, $\left\{ V^i \right\}_{i=1}^{n}\subset C\Delta(X)$, $f:\mathbb{R}^n\to\mathbb{R}$ is differentiable. Let $D\equiv\mathbb{R}^n$. Then a necessary condition for $P^*$ solving \cref{eqn:P} is:
	\begin{align}
		P^*\in\arg\max_{\substack{P\in \Delta^2(X)\\E_P[\nu]=\mu}}\nabla f(E_{P^*}[V^1],\dots,E_{P^*}[V^n])\cdot\left( E_P[V^1],\dots,E_P[V^n] \right)\label{eqn:foc}
	\end{align}
	\label{thm:foc}
\end{thm}
\begin{proof}
	Solving \cref{eqn:P} is equivalent to solving \cref{eqn:reduced}. Suppose by contradiction that \cref{eqn:foc} is violated at optimal $P^*$. Then it is equivalently saying that there exists $v\in \mathcal{V}(\mu)$ such that:
	\begin{align*}
		\nabla f(v^*)\cdot v^*<\nabla f(v^*)\cdot v
	\end{align*}
	By \cref{lem:1}, $\mathcal{V}(\mu)$ is a convex set. Therefore $v_{\alpha}=(1-\alpha)v^*+\alpha v\in \mathcal{V}(\mu)$. Consider $h(\alpha)=f(v_{\alpha})$. Then $h'(0)=\nabla f(v^*)\cdot(v-v^*)>0$. So there exists $\alpha'>0$ s.t. $h(\alpha')>h(0)$. Then $f(v^*)< f(v_{\alpha'})$. Contradicting optimality of $v^*$.
\end{proof}

\begin{thm}
	Let $X$ be a non-empty state space, $\left\{ V^i \right\}_{i=1}^{n+m}\subset C\Delta(X)$, $f:\mathbb{R}^{n+m}\to\mathbb{R}$ is constant in last $m$ arguments. Let $D\equiv\left\{ v|v^i\ge 0\ \forall i>n \right\}$. Then there exists $P^*$ solving \cref{eqn:P} and $\lambda\in B_{m+n}$ such that:
	\begin{align*}
		P^*\in\arg\max_{\substack{P\in \Delta^2(X)\\E_P[\nu]=\mu}}E_{P}\left[ \sum\lambda^i V^i \right]
	\end{align*}
	\label{thm:lag}
\end{thm}
\begin{proof}
	$\forall P^*$ solving \cref{eqn:P}, let $v^*$ be corresponding value. Define:
	\begin{align*}
		v_{\alpha}=v^*+\alpha(\underbrace{0,\dots,0}_{n},\underbrace{1,\dots,1}_m)
	\end{align*}
	Then by definition $f\left( v_{\alpha} \right)=f(v^*)$. $v_0=v^*\in \mathcal{V}(\mu)$. Since $\mathcal{V}(\mu)$ is bounded, for large enough $\alpha$, $v_{\alpha}\not\in \mathcal{V}(\mu)$. Then since $\mathcal{V}(\mu)$ is compact, there exists $\alpha$ s.t. $v_{\alpha}\in \partial\mathcal{V}(\mu)$. Since $\mathcal{V}(\mu)$ is convex, there exists $l\in L(\mathbb{R}^{m+n})$ s.t. $v_{\alpha}\in\arg\max_{v\in \mathcal{V}(\mu)} l(v)$. Let $l=\sum \lambda^i v^i$, then:
	\begin{align*}
		v_{\alpha}\in\arg\max_{v\in \mathcal{V}(\mu)}\sum \lambda^i v^i
	\end{align*}
	Let $P_{\alpha}$ be the corresponding information structure implementing $v_{\alpha}$ (existence of $P_{\alpha}$ guaranteed by \cref{lem:1}). Then
	\begin{align*}
		P_{\alpha}\in\arg\max_{\substack{P\in \Delta^2(X)\\E_P[\nu]=\mu}}E_{P}\left[ \sum\lambda^i V^i \right]
	\end{align*}
	Since $f(v_{\alpha})=f(v^*)$, $P_{\alpha}$ solves \cref{eqn:P} as well.
\end{proof}

\subsection{Convex optimization}

\begin{thm}
	Let $X$ be a non-empty state space, $\left\{ V^i \right\}_{i=1}^n\subset C\Delta(X)$, $D\equiv\left\{ v|g(v)\ge 0 \right\}$. If both $f$ and $g$ are quasi-concave and continuous, then there exists $P^*$ solving \cref{eqn:P}, $v^*=(E_P[V^i])$  and $\lambda\in B_n$ such that:
	\begin{align*}
		\begin{dcases}
			P^*\in\arg\max_{\substack{P\in \Delta^2(X)\\E_P[\nu]=\mu}}E_{P}\left[ \sum\lambda^i V^i \right]\\
			v^*\in\arg\min_{f(v)\ge f(v^*),v\in D}\lambda\cdot v
		\end{dcases}
	\end{align*}
	\label{thm:convex}
\end{thm}

\begin{proof}
	First, by \cref{thm:1}, $P^*$ solving \cref{eqn:P} exists. Then by optimality of $P^*$: $$\mathcal{V}(\mu)\bigcap \left\{ v|v\in D,f(v)>f(v^*) \right\}=\emptyset$$
	Since $f$ and $g$ are quasi-convex, $\left\{ v|v\in D,f(v)> f(v^*) \right\}$ is a convex set. Then by separating hyperplane theorem, there exists $c$ and $\lambda$ s.t. $\forall v\in \mathcal{V}(\mu)$, $v'\in D$ and $f(v')> f(v^*)$:
	\begin{align*}
		\lambda\cdot v \le c\text{ and }\lambda\cdot v'>c
	\end{align*}
	By continuity of $f$ and $g$, $v^*\in \mathrm{cl}\left( \left\{ v|v\in D,f(v)> f(v^*) \right\} \right)$. So $\lambda\cdot v^*=c$. Then it is easy to verify that $\lambda$ satisfies the conditions in \cref{thm:convex}.
\end{proof}

\begin{cor}
	\label{cor:convex}
		Let $X$ be a non-empty state space, $\left\{ V^i \right\}_{i=1}^n\subset C\Delta(X)$, $f:\mathbb{R}^n\to \mathbb{R}$ is quasi-concave. Let $D\equiv\left\{ v|g(v)\ge 0 \right\}$, $g$ is quasi-concave. If $f$ and $g$ are both differentiable, then there exists $P^*$ solving \cref{eqn:P}, $v^*=(E_P[V^i])$ and $\gamma,\eta\ge 0$ such that:
	\begin{align*}
			P^*\in\arg\max_{\substack{P\in \Delta^2(X)\\E_P[\nu]=\mu}}\left( \eta\nabla f(v^*)+\gamma \cdot \mathrm{J}g(v^*) \right)\cdot \left( E_P[V^1],\cdots,E_P[V^n] \right)\\
	\end{align*}
\end{cor}
\begin{proof}
	By \cref{thm:convex}:
	\begin{align}
		v^*\in \arg\min_{f(v)\ge f(v^*),v\in D}\lambda\cdot v\label{eqn:P:dual}
	\end{align}
	It is easy to verify that \cref{eqn:P:dual} as a dual problem is a convex optimization problem. Since both $f$ and $g$ are differentiable, by Kuhn-Tucker condition, there exists $\gamma,\eta\ge 0$ such that:
	\begin{align*}
		\lambda-\eta\cdot\nabla f(v^*)-\gamma\cdot \mathrm{J}g(v^*)=0
	\end{align*}
	Then by definition of $\lambda$:
	\begin{align*}
			P^*\in\arg\max_{\substack{P\in \Delta^2(X)\\E_P[\nu]=\mu}}\left( \eta\nabla f(v^*)+\gamma \cdot \mathrm{J}g(v^*) \right)\cdot \left( E_P[V^1],\cdots,E_P[V^n] \right)\\
	\end{align*}
\end{proof}

\subsection{Maximum theorem}
\begin{thm}
	Let $X$ be a non-empty state space, $\left\{ V^i \right\}_{i=1}^{n}\subset C\Delta(X)$, $f\in C\mathbb{R}^n$. Suppose $D(\mu)$ is a continuous correspondence and $\forall \mu\in \Delta(X)$, $\mathcal{V}(\mu)\bigcap D(\mu)\neq \emptyset$. Let $\kappa(\mu)$ be maximum in \cref{eqn:P} and $\mathcal{P}(\mu)$ be solution to $\cref{eqn:P}$, then $\kappa(\mu)$ is continuous and $\mathcal{P}(\mu)$ is compact-valued and upper hemicontinuous\footnote{with respect to Prokhorov metric.}.
	\label{thm:maximum}
\end{thm}
\begin{proof}
	\cref{thm:maximum} is an application of maximum theorem. Since by \cref{lem:closegraph} $\mathcal{V}(\mu)$ and $D(\mu)$ are both continuous, $\mathcal{V}(\mu)\bigcap D(\mu)$ is non-empty, compact valued and continuous. \cref{eqn:P} is equivalent to maximizing $f(v)$ on $\mathcal{V}(\mu)\bigcap D(\mu)$. Therefore, by maximum theorem, $\kappa(\mu)$ is continuous and the argmax correspondence $V^*(\mu)$ is comapct-valued and upper hemicontinuous.\par
	Now we show that $\mathcal{P}(\mu)$ is compact valued and upper hemicontinuous.
	\begin{itemize}
		\item compactness: (sequential comapctness will be sufficient) $\forall \left\{ P_m \right\}\subset \mathcal{P}(\mu)$, consider $v_m=E_{P_m}\left[ (V^i) \right]$. Then $v_m\in V^*(\mu)$, so there exists subsequence (without loss assume to be $v)m$ itself) $v_{m}\to v\in V^*(\mu)$. Then since $\Delta^2(X)$ is compact by \cref{thm:prok}, there exists subsequence $P_m\xrightarrow{w-^*} P\in \Delta^2(X)$. Then $E_{P}[(V^i)]=\lim E_{P_m}[(V^i)]=\lim v_m=v\in V^*(\mu)$. So $P\in \mathcal{P}(\mu)$.
		\item upper hemicontinuity: $\forall \mu_m\to \mu$, $P_m\xrightarrow{w-^*}P$ and $P_m\in \mathcal{P}(\mu_m)$. Then $v_m=E_{P_m}[(V^i)]\in V^*(\mu_m)$. By definition of w-$^*$ convergence, $v_m\to v=E_P[(V^i)]$. By upper hemicontinuity of $V^*(\mu)$, $v\in V^*(\mu)$. Therefore, $P\in \mathcal{P}(\mu)$.
	\end{itemize}
\end{proof}

\section{Applications}
\label{sec:app}
\subsection{Costly Information acquisition}
A direct application of \cref{thm:1} is to costly information acquisition problem. Consider a variant of rational inattention model. Decision utility at each belief is $F(\mu)=\max_{a}E_{\mu}[u(a,x)]$. Information measure of any experiment $P$ is $I(P|\mu)=E_P[H(\mu)-H(\nu)]$ where $H$ is the entropy function. Assume cost of experiments are convex in their measure, the decision problem can be written as:
\begin{align}
	\sup_{\substack{P\in D^2(X)\\E_P[\nu]=\mu}}E_P\left[ F(\nu) \right]-f\left( E_P\left[ H(\mu)-H(\nu) \right] \right)
	\label{eqn:RI}
\end{align}
In a standard rational inattention problem, $f$ is linear. Then standard concavification method suggests that optimal experiment involves signals no more than $\abs{X}$. The reason why we want to deviate from a linear $f$ is that standard RI has two kind of debatable predictions: 1) prior invariant choice of optimal posteriors (see \cite{caplin2013rational}). 2) no dynamics if we allow repeated experiments (see \cite{steiner2017rational}). However, when $f$ is more general, say convex, we knew little about how to solve \cref{eqn:RI}. \cref{thm:foc} becomes useful. 
\begin{prop}
	There exists $P^*$ solving \cref{eqn:RI}, $\abs{\mathrm{supp}(P^*)}=2\abs{X}$. Moreover, if $f$ is differentiable, $P^*$ solves:
	\begin{align*}
		P^*\in\arg \max_{\substack{P\in\Delta^2(X)\\E_P[\nu]=\mu}} E_{P}\left[ F(\nu)-f'\left( E_{P^*}[H(\mu)-H(\nu)] \right)\cdot H(\nu) \right]
	\end{align*}
	\label{prop:RI}
\end{prop}

\subsection{Dynamic information design}
Consider the following Bellman equation:
\begin{align}
	V(\mu)=\max\bigg\{ F(\mu), \sup_{P\in \Delta^2(X)} &e^{-\rho dt} E_{P}[V(\nu)]-f\left( E_P[H(\mu)-H(\nu)] \right) \bigg\}\label{eqn:dri}\\
	\mathrm{s.t.}\ &
	\begin{dcases}
		E_P[\nu]=\mu\\
		E_P[H(\mu)-H(\nu)]\le C
	\end{dcases}\notag
\end{align}

\begin{prop}
	If $F,H\in C\Delta(X)$, $f\in C\mathbb{R}$. $F(x),f(x),C\ge0$. Then there exists unique $V\in C\Delta(X)$ solving \cref{eqn:dri}.
	\label{prop:dri}
\end{prop}
\begin{proof}
	Let $\mathcal{Z}=\left\{ V\in C\Delta(X)|F\le V\le \mathrm{co}(F) \right\}$. We define operator:
	\begin{align}
		T(V)(\mu)=\max\bigg\{ F(\mu), \sup_{P\in \Delta^2(X)} &e^{-\rho dt} E_{P}[V(\nu)]-f\left( E_P[H(\mu)-H(\nu)] \right) \bigg\}\label{eqn:contract}\\
	\mathrm{s.t.}\ &
	\begin{dcases}
		E_P[\nu]=\mu\\
		E_P[H(\mu)-H(\nu)]\le C
	\end{dcases}\notag
	\end{align}
	By \cref{thm:1}, the max operator is well defined. When $P=\delta_{\mu}$, $E_P[\nu]=\mu$ and $E_P[H(\mu)-H(\nu)]=0$ so the sup operator is also well defined. Now we prove that $T$ is a contraction mapping on $(\mathcal{Z},L_{\infty})$.
	\begin{itemize}
		\item $T(\mathcal{Z})\subset \mathcal{Z}$: First of all, given the outter max operator in \cref{eqn:contract}, $T(V)(\mu)\ge F(\mu)$. Then $\forall P\in \Delta^2(X)$ such that $E_P[\nu]=\mu$ and $E_P[H(\mu)-H(\nu)]\le C$:
			\begin{align*}
				&e^{-\rho dt} E_{P}[V(\nu)]-f\left( E_P[H(\mu)-H(\nu)] \right)\\
				\le&e^{-\rho dt} E_P[V(\nu)]\\
				\le& E_{P}[\mathrm{co}(F)(\nu)]\\
				=&\mathrm{co}(F)(\mu)
			\end{align*}
			First inequality is from $f$ being non-negative, second inqeuality is from $V$ being non-negative, $e^{-\rho dt}<1$ and $V\le \mathrm{co}(F)$. Last equality is from $\mathrm{co}(F)$ being linear. Last step is to show $T(\mathcal{Z})(\mu)\in C\Delta(X)$. This is directly implied by \cref{thm:maximum}.
		\item $T(V)$ is monotonic: Suppose $U(\mu)\ge 0$ and $U+V\in \mathcal{Z}$ If $T(V)(\mu)=F(\mu)$, then by construction $T(V+U)\ge F(\mu)=T(V)(\mu)$. If $T(V)(\mu)>F(\mu)$, let $P$ be solution to \cref{eqn:contract} at $\mu$ for $V$:
			\begin{align*}
				T(V+U)(\mu)\ge& e^{-\rho dt} E_P[V(\nu)+U(\nu)]-f\left( E_P[H(\mu)-H(\nu)] \right)\\
				=&T(V)(\mu)+e^{-\rho dt}E_P[U(\nu)]\\
				\ge&T(V)(\mu)
			\end{align*}
			And constraints $E_P[H(\mu)-H(\nu)]\le C$ and $E_P[\nu]=\mu$ are independent of choice of $V$ so still satisfied.
		\item $T(V)$ is contraction. We claim that $T(V+\alpha)(\mu)\le T(V)(\mu)+e^{-\rho dt}\alpha$. Suppose not true at $\mu$. Obviously $T(V+\alpha)(\mu)>F(\mu)$. Then let $P$ be the solution of \cref{eqn:contract} at $\mu$ for $V+\alpha$.
			\begin{align*}
				T(V)(\mu)\ge& e^{-\rho dt} E_P[V(\nu)]-f\left( E_P[H(\mu)-H(\nu)] \right)\\
				=&e^{-\rho dt} E_P[V(\nu)+\alpha]-f\left( E_P[H(\mu)-H(\nu)] \right)-e^{-\rho dt} \alpha\\
				=&T(V+\alpha)(\mu)-e^{-\rho dt}\alpha\\
				>& T(V)(\mu)
			\end{align*}
			Similar to last part,  constraints $E_P[H(\mu)-H(\nu)]\le C$ and $E_P[\nu]=\mu$ are still satisfied. Contradiction.
	\end{itemize}
	Therefore, by Blackwell condition, $T(V)$ is a contraction mapping on $\mathcal{Z}$. There exists a unique solution $V\in\mathcal{Z}$ solving the fixed point problem $T(V)=V$. 
\end{proof}

\subsection{Persuade voters with outside options}
Consider a politician who can strategically design a public signal to voters to influence their voting behavior (the setup in \cite{10.1257/aer.20140737}).\par
\emph{Voting game}: There are $n\ge 1$ voters who chooses from a binary policy set $A=\left\{ a_0,a_1 \right\}$. There are two states $X= \left\{ x_0,x_1 \right\}$. Each voter gets Bernoulli utility $u_i(a,x)$ from voting for the policy $a$. Assume that $a_1$ is unanimously preferred to $a_0$ when $x_1$ is the true state and vice versa. The politician has state independent utility over policies and prefers $a_1$ strictly to $a_0$. I assume that $a_0$ is a default policy. For $a_1$ to be proved, the politician needs more than $m$ $(m\le n)$ voters to voter for $a_1$. The politician can design a signal structure to influence voters' decisions. Equivalently, I assume that the politician chooses a distribution over posterior beliefs $P\in \Delta^2(X)$.\par
\emph{Outside option}: Different from \cite{10.1257/aer.20140737}, where number of potential voters is fixed, I assume that each voter has opportunity cost $c_i$ of participating in the voting game. Therefore, to approve the new policy, the politician should first attract at least $m$ voters to the game and then persuade them to vote for $a_1$.\par
To simplify notation, I write all voter's utility as functions of belief $F_i(\mu)$. Let $\overbar{\mu}_i$ be the threshold belief for each voter to vote for $a_1$ The politician's optimization problem can be written as:
\begin{align}
	\label{eqn:vote}
	\sup_{i_1,\dots, i_k, P}&E_{P}\left[ \mathbf{1}_{\#(\mu\ge \overbar{\mu}_{i_j})\ge m} \right]\\
	\mathrm{s.t.}\ &
	\begin{dcases}
		E_{P}[F_{i_j}]\ge c_{i_j}\\
		E_{P}[\nu]=\mu
	\end{dcases}\notag
\end{align}
Notice that in \cref{eqn:vote}, the politician doesn't necessarily need to exclude voters outside of $\left\{ i_1,\dots,i_k \right\}$, so the maximum from \cref{eqn:vote} must be weakly larger than the politician's optimal utility. On the other hand, for any strategy in \cref{eqn:vote}, potentially including more voters to the voting game can only make the politician better off. So \cref{eqn:vote} exactly characterizes the politician's optimization problem.\par
For any voter, except for $\overbar{\mu}_i$, there is another critical belief $\widetilde{\mu}_i$:
\begin{align*}
	\frac{\widetilde{\mu}_i-\mu}{\widetilde{\mu}_i} F_i(0)+\frac{\mu}{\widetilde{\mu}_i}F_i(\widetilde{\mu}_i)=c_i
\end{align*}
Suppose voter observes information structure inducing posterior belief $0$ and $\widetilde{\mu}_i$, then the voter is exactly indifferent between paying the opportunity cost and entering the voting game and not.
\begin{prop}
	\label{prop:vote}
	Let $\mu^*$ be the smallest belief s.t. $\#\left\{ i|\overbar{\mu}_i\ge \mu^* \right\}\ge m$ and $\#\left\{ i|\widetilde{\mu}_i\ge\mu^* \right\}\ge m$, then the optimal strategy for \cref{eqn:vote} is:
	\begin{align*}
		\begin{cases}
			P(0)=\frac{\mu^*-\mu}{\mu^*}\\
			P(\mu^*)=\frac{\mu}{\mu^*}
		\end{cases}
	\end{align*}
	and $\left\{ i_1,\dots,i_k \right\}=\left\{ i|\min\left\{ \widetilde{\mu}_i,\overbar{\mu}_i \right\}\ge\mu^* \right\}$.
\end{prop}
\cref{prop:vote} states that when voters must pay opportunity cost to enter the voting game, then there are potentially two pivotal voters. One is the one who's most difficult to persuade to adopt $a_1$, and the other is the one who's most difficult to attract to the voting game. Both \emph{difficulty} levels are measured by the location of the critical beliefs.

\begin{proof}
	The key step of proving \cref{prop:vote} is to apply \cref{cor:convex} to \cref{eqn:vote}. Notice that the objective function is \cref{eqn:vote} is in fact an indicator function with some threshold belief level (say $\mu'$, which is the lowest belief to persuade at least $m$ voters to vote for $a_1$). So \cref{cor:convex} is directly applicable to \cref{eqn:vote}, and the objective function is in the form of:
	\begin{align}
		\sum \lambda_{i_j} \max\left\{ 0,\mu-\overbar{\mu}_{i_j} \right\}+\mathbf{1}_{\mu\ge\mu'}\label{eqn:vote:lag}
	\end{align}
	It is easy to see that \cref{eqn:vote:lag} is a convex function on $\mu\in[0,\mu']$ and a linear function on $\mu\in[\mu',1]$ (there is no point to include voters who will never vote for $a_1$.). So optimal persuasion strategy must induce either belief $0$ or interior belief $\nu>\mu'$. Of course since at least $m$ voters are included and persuaded, $\nu\ge\mu^*$. On the other hand, it is easy to verify that the strategy define by $\mu^*$ induces at least $m$ voters to participate, so $\mu^*$ is optimal.
\end{proof}

\subsection{Screening with information}
Consider a problem of Bayesian persuasion with unknown receiver types. Let $\Theta$ be the set of receiver types, $X$ be the finite set of states and $A$ be the set of actions. $\forall \theta\in \Theta$, decision utility at each belief is $F_{\theta}(\mu)=\max_{a}E_{\mu}[u(a,x,\theta)]$. Sender's utility at each belief given receiver type $\theta$ is $V_{\theta}(\mu)$. Assume that the type distribution is $\pi(\theta)\in \Delta(\Theta)$. The sender can screen the receivers by providing a menu of information structures. Then by revelation principle, sender's optimization problem is:
\begin{align}
	\sup_{P_{\theta}\in \Theta\times \Delta^2(X)}&\int E_{P_{\theta}}[V_{\theta}]\pi(\theta)\mathrm{d} \theta\label{eqn:screen}\\
	\mathrm{s.t.}\ &
	\begin{dcases}
	E_{P_{\theta}}[F_{\theta}]\ge E_{P_{\theta'}}[F_{\theta}]\ \forall \theta,\theta'\in \Theta \\
	E_{P_{\theta}}[\nu]=\mu\ \forall \theta\in \Theta
\end{dcases}\notag
\end{align}
When $\Theta$ and $A$ are both infinite, solving \cref{eqn:screen} is difficult due to the dimensionality of strategy space. When $A$ is finite, it is WLOG to restrict the sender to use direct message which suggests the actions being played conditional on the state. Then \cref{eqn:screen} reduces to a screening problem with finite dimensional strategy function (plus a few more obedience constraints). In the remaining case where $\Theta$ is finite but $A$ is infinite, it is still unclear whether it is WLOG to consider only finite dimensional screening mechanisms.\par
Now consider the finite $\Theta$ case. Suppose $\Theta=\left\{ 1,\dots, N \right\}$. Define:
\begin{align*}
	\begin{dcases}
	\mathcal{V}_i(\mu)=\left\{ E_P[V_i],E_P[F_1],\dots,E_P[F_N]\bigg|P\in \Delta^2(X),E_P[\nu]=\mu \right\}\\
	D(\mu)=\left\{ v\in \mathbb{R}^{(N\times (N+1))}\Big| v_{i}^{i+1}\ge v_{j}^{i+1}\forall i,j \right\}
\end{dcases}
\end{align*}
Then \cref{eqn:screen} is equivalent to the following problem:
\begin{align}
	\sup_{v\in D(\mu)\bigcap \times_{i=1}^N\mathcal{V}_i(\mu)} \pi_i v_i^1
	\label{eqn:screen:1}
\end{align}
By \cref{lem:1}, each $\mathcal{V}_i(\mu)$ is compact set. Therefore, $D(\mu)\bigcap \times \mathcal{V}_i(\mu)$ is compact. It is easy to see that $D(\mu)\bigcap \times \mathcal{V}_i(\mu)$ is non-empty. By Wierestrass's theorem, there exists $v^*$ solving \cref{eqn:screen:1}. Then by \cref{lem:1}, there exists $P^*_i\in \Delta^2(X)$ s.t. $v^{*}_i=\left( E_{P^*_i}[V_i],E_{P^*_i}[F_1],\dots,E_{P^*_i}[F_N] \right)$ and $\abs{\mathrm{supp}(P^*_i)}\le (N+2)\cdot \abs{X}$. Therefore, $\left( P^*_1,\dots,P^*_N \right)$ solves \cref{eqn:screen} and we get the following proposition:
\begin{prop}
	\label{prop:screen}
	If $\Theta$ is finite, then $\forall \mu\in \Delta(X)$, there exists $\left( P^*_1,\dots, P^*_N \right)\in \Delta^2(X)^N$ solving \cref{eqn:screen} and each $\abs{\mathrm{supp}(P^*_i)}\le (N+2)\cdot \abs{X}$.
\end{prop}

\cref{prop:screen} states that it is WLOG to consider only mechanisms with finite support when solving \cref{eqn:screen}. Therefore, it is sufficient to maximize over $N(N+2)\cdot \abs{X}$ posterior beliefs and $N(N+2)\cdot \abs{X}$ corresponding probabilities to solve constrained optimization problem \cref{eqn:screen}, which is a computationally tractable problem.

\section{Conclusion}

In this paper, I study the set of all possible combinations of expected valuations that can be implemented by designing information. I show that the set can be implemented only using information structures with finite realizations, and all extreme points of the set can be characterized using a concavification characterization. I developed a Lagrange method in the information design setup, and applied the results to various applications including static and dynamic information acquisition, persuasion of receivers with outside options and screening using information.

\newpage
\small
\bibliographystyle{apalike}
\bibliography{report}

\begin{thebibliography}{}

\bibitem[Alonso and Câmara, 2016]{10.1257/aer.20140737}
Alonso, R. and Câmara, O. (2016).
\newblock Persuading voters.
\newblock {\em American Economic Review}, 106(11):3590--3605.

\bibitem[Aumann et~al., 1995]{aumann1995repeated}
Aumann, R.~J., Maschler, M., and Stearns, R.~E. (1995).
\newblock {\em Repeated games with incomplete information}.
\newblock MIT press.

\bibitem[Billingsley, 2013]{billingsley2013convergence}
Billingsley, P. (2013).
\newblock {\em Convergence of probability measures}.
\newblock John Wiley \& Sons.

\bibitem[Caplin and Dean, 2013]{caplin2013rational}
Caplin, A. and Dean, M. (2013).
\newblock Behavioral implications of rational inattention with shannon entropy.
\newblock Working Paper 19318, National Bureau of Economic Research.

\bibitem[Carath{\'e}odory, 1907]{caratheodory1907variabilitatsbereich}
Carath{\'e}odory, C. (1907).
\newblock {\"U}ber den variabilit{\"a}tsbereich der koeffizienten von
  potenzreihen, die gegebene werte nicht annehmen.
\newblock {\em Mathematische Annalen}, 64(1):95--115.

\bibitem[Kamenica and Gentzkow, 2011]{kamenica2011bayesian}
Kamenica, E. and Gentzkow, M. (2011).
\newblock Bayesian persuasion.
\newblock {\em The American Economic Review}, 101(6):2590--2615.

\bibitem[Rockafellar, 1969]{rockafellar1970convex}
Rockafellar, R. (1969).
\newblock {\em Convex analysis}.
\newblock Princeton university press.

\bibitem[Rudin, 1991]{rudin1991functional}
Rudin, W. (1991).
\newblock Functional analysis 2nd ed.
\newblock {\em International Series in Pure and Applied Mathematics.
  McGraw-Hill, Inc., New York}.

\bibitem[Steiner et~al., 2017]{steiner2017rational}
Steiner, J., Stewart, C., and Mat{\v{e}}jka, F. (2017).
\newblock Rational inattention dynamics: Inertia and delay in decision-making.
\newblock {\em Econometrica}, 85(2):521--553.

\bibitem[Straszewicz, 1935]{straszewicz1935exponierte}
Straszewicz, S. (1935).
\newblock {\"U}ber exponierte punkte abgeschlossener punktmengen.
\newblock {\em Fundamenta Mathematicae}, 24:139--143.

\end{thebibliography}

\appendix
\section{Theorems used in proof}
Here are key theorems used for my proof. \cref{thm:stra} is Straszewicz's theorem (\cite{straszewicz1935exponierte}, see Theorem 18.6 of \cite{rockafellar1970convex}). \cref{thm:krein} is Krein-Milman theorem(see Theorem 3.23 of \cite{rudin1991functional}). \cref{thm:cara} is Carath{\'e}odory's theorem (\cite{caratheodory1907variabilitatsbereich}).\cref{thm:prok} is Prokhorov's theorem (see Theorem 5.1 of \cite{billingsley2013convergence})
\begin{thm}
	Let $C\in\mathbb{R}^n$ be a closed convex set, $\mathrm{cl}\left( \mathrm{exp}(C) \right)=\mathrm{ext}(C)$. 
	\label{thm:stra}
\end{thm}

\begin{thm}
	Let $C\in\mathbb{R}^n$ be a compact and convex set, $C=\mathrm{conv}(\mathrm{ext}(C))$. 
	\label{thm:krein}
\end{thm}

\begin{thm}
	Let $C\in\mathbb{R}^n$, if $x\in \mathrm{conv}(C)$ then $x\in \mathrm{conv}(R)$ for $R\subset C$, $\abs{R}\le n+1$. 
	\label{thm:cara}
\end{thm}

\begin{thm}
	A tight set $\Pi$ of probability measures on Borel sets of metric topological space $\mathcal{X}$ is relative compace in weak-$^*$ topology.
	\label{thm:prok}
\end{thm}

\begin{lem}
	Let $C$ be a convex set in $\mathbb{R}^n$. Then $\forall F\in F(C)$, $\mathrm{ext}(F)\subset \mathrm{ext}(C)$.
	\label{lem:face}
\end{lem}
\begin{proof}
	$\forall x\in \mathrm{ext}(F)$ there exists affine $f$ defining face $F$. $\forall y,z\in C$. Suppose $y\in F$, then $f(x)=f(y)$. If there exists $\alpha\in(0,1)$ s.t. $\alpha y+(1-\alpha)z=x$, then $\alpha f(y)+(1-\alpha) f(z)=f(x)\implies\ f(z)=f(x)=f(y)$ so $z\in F$. Since $x\in \mathrm{ext}(F)$, $x\in\left\{ y,z \right\}$. Suppose $y\not\in F$, then $f(x)=\alpha f(y)+(1-\alpha) f(z)< f(x)$ by definition of $f$, contradiction. To sum up, $x\in \mathrm{ext}(C)$.
\end{proof}

\end{document}